\documentclass[egregdoesnotlikesansseriftitles,a4paper,11pt, leqno,bibliography=totoc]{article}

\usepackage[english]{babel}			
\usepackage[utf8]{inputenc} 		
\usepackage[T1]{fontenc}		
\usepackage{lastpage}
\usepackage{hyperref}						
\usepackage{mathtools}
\usepackage{amsthm}
\usepackage{amssymb}
\usepackage{graphicx}
\usepackage{stackengine}
\usepackage[noadjust]{cite}
\usepackage{authblk}
\usepackage{abstract}
\usepackage[all,defaultlines=3]{nowidow}
\usepackage{enumitem}
\mathtoolsset{showonlyrefs, showmanualtags}
\stackMath

\theoremstyle{plain}
\newtheorem{theorem}{Theorem}[section]
\newtheorem*{theorem*}{Theorem}
\newtheorem{lemma}[theorem]{Lemma}
\newtheorem*{lemma*}{Lemma}

\newtheorem*{korollar*}{Korollar}

\newtheorem*{proposition*}{Proposition}

\newtheorem*{satz*}{Satz}

\newtheorem*{corolarry*}{Corollary}

\newtheoremstyle{named}{\topsep}{\topsep}{\itshape}{0pt}{\bfseries}{}{5pt plus 1pt minus 1pt}{\thmname{#1}\thmnumber{ #2.}\thmnote{\;(#3)}}
\theoremstyle{named}
\newtheorem*{namedtheorem*}{Theorem}

\newtheorem*{namedlemma*}{Lemma}

\newtheorem*{namedkorollar*}{Korollar}

\newtheorem*{namedproposition*}{Proposition}

\newtheorem*{namedcorolarry*}{Corollary}

\newtheorem{propositionA}{Proposition}

\theoremstyle{definition}

\newtheorem*{definition*}{Definition}

\newtheorem*{beispiel*}{Beispiel}

\newtheorem*{bemerkung*}{Bemerkung}
\newtheorem*{erinnerung*}{Erinnerung}

\newtheorem*{remark*}{Remark}

\newtheorem*{assumption*}{Assumption}

\newtheorem*{notation*}{Notation}

\addto\captionsenglish{\renewcommand{\proofname}[2][]{Proof{#1}{#2}. }}
\renewenvironment{proof}[1][]{{\noindent\textbf{\proofname{#1}}}}{\qed\\}

\renewcommand{\L}{\mathcal{L}}

\newcommand{\R}{\mathbb{R}} 	
\newcommand{\N}{\mathbb{N}} 	

\newcommand{\norm}[1]{\left\lVert#1\right\rVert}

\newcommand{\grad}{\nabla}
\renewcommand{\div}{\grad\cdot}
\newcommand{\mel}{\MoveEqLeft}

\newcommand{\termzeta}{\frac{1+\left|\nabla'\zeta\right|^2}{\left(\partial_n\zeta\right)^2}}
\newcommand{\termw}{\frac{1+\left|\nabla'w\right|^2}{\left(1+\partial_nw\right)^2}}

\newcommand{\rest}[1]{R_{#1}\left(\nabla w\right)}

\newcommand{\eps}{\varepsilon}

\begin{document}

\title{Stability of traveling waves for doubly nonlinear equations}

\author{Christian Seis\footnote{Institut f\"ur Analysis und Numerik,  Universit\"at M\"unster. Email: seis@uni-muenster.de} \and Dominik Winkler\footnote{Institut f\"ur Analysis und Numerik,  Universit\"at M\"unster. Email: dominik.winkler@uni-muenster.de}}
			

	\maketitle
     
     \begin{abstract}
          In this note, we investigate a doubly nonlinear diffusion equation in the slow diffusion regime. We prove stability of the pressure of solutions that are close to traveling wave solutions in a homogeneous Lipschitz sense. 
          We derive regularity estimates for arbitrary derivatives of the solution's pressure by extending existing results for the  porous medium equation \cite{Kienzler16}.
     \end{abstract}
\section{Introduction and Main Theorem}

We consider the doubly nonlinear degenerate parabolic equation
\begin{align}\label{1}
    \partial_\tau \rho - \div \left(\rho^{m-1} \left|\nabla \rho\right|^{p-2}\nabla \rho \right) = 0 \quad \text{ in } \left[0,\infty\right) \times \R^n,
\end{align}
with nonlinearity exponents
\begin{align}\label{2}
    p>1 \quad \text{ and } \quad m+p > 3. 
\end{align}
Equations of this type are used to describe the turbulent filtration of a fluid through a porous medium  \cite{Leibenson1945,Leibenson45}, or the filtration of a non-Newtonian fluid \cite{Kalashnikov1987,WuZhaoYinLi2001}.
The variable $\rho$ measures the fraction of the volume that is filled by the fluid. 
Thereby the parameter $p$ prescribes the turbulence of the flow arising from a nonlinear (for $p\neq2$) version of Darcy's law. The quantity $\frac{m+p-3}{p-1}$ dictates the relation between the pressure and the density of a polytropic fluid via the relation
\[
\text{pressure}\sim \text{density}^{\frac{m+p-3}{p-1}}.
\]
 For further information about the physical and modeling background we refer to \cite{Antontsev2002,Kalashnikov1987,WuZhaoYinLi2001}. 

Equation \eqref{1} is referred to as the doubly nonlinear diffusion equation as it merges the nonlinear effects of the porous medium equation (case $p=2$) and the parabolic $p$-Laplace equation (case $m=1$). Indeed, \eqref{1} can be rewritten (after a linear rescaling of time $\tau \mapsto q^{1-p}\tau$) as 
\begin{align}
    \partial_\tau \rho -\Delta_p\left(\rho^q\right) = 0 \tag{\ref{1}$^\prime$},
\end{align}
where $\Delta_p$ denotes the $p$-Laplacian and $q \coloneqq \frac{m-1}{p-1}+1$ satisfies $q>\frac{1}{p-1}$ according to \eqref{2}.

The condition \eqref{2} determines the so-called slow diffusion regime: The diffusion flux $\rho^{m-1}\left|\nabla \rho \right|^{p-2}$ degenerates if $\rho$ vanishes and thus, a compactly supported solution propagates with finite speed and remains compactly supported for all later times, see for example \cite{Kalashnikov1987}. 
Due to the finite propagation speed, the problem features a free moving boundary $\partial \mathcal{P}(\rho(t))$, where $\mathcal{P}(\rho(t)) = \overline{\left\{y\in \R^n : \rho(y,t)> 0\right\}}$ is the support of the solution. The asymptotic flatness of this free boundary, its regularity and the regularity of solutions near the boundary are our main concerns  in this paper.

Existence results for (weak) solutions of \eqref{1} can be found in \cite{Bernis1988,Lions1969,SturmStefan2017}. H\"older continuity of solutions and some Harnack type inequalities were proved in \cite{Ivanov1997,PorzioVespri1993} and \cite{KinnunenKuusi2007,FornaroSosio2008,Vespri1994}. See also \cite{Kalashnikov1987} for a survey of classical results regarding (even more general) degenerate parabolic equations.

In this work, we examine nonnegative solutions of \eqref{1} that are close to a traveling wave solution in order to get insights in the qualitative behavior of general solutions. As we are particularly interested in the behavior of the free boundary, considering traveling wave solutions appears to be natural. Indeed, traveling wave solutions exhibit the same boundary behavior as the so-called Barenblatt solutions, which are the radially symmetric self-similar solutions for \eqref{1} constructed first in \cite{Barenblatt1952}. These self-similar solutions do, similar to other parabolic problems, prescribe the large-time asymptotic behavior of any solution with finite mass, see \cite{Agueh03,AguehBlanchetCarillo10}.
Moreover, it can be shown that any compactly supported solution of \eqref{1} has the same spreading behavior as the Barenblatt solutions \cite{TedeevVespri2015}.
Accordingly, localizing around an arbitrary boundary point of any solution would result to leading order  in the problem on the half space we are considering here.

Traveling waves display the typical behavior of solutions also in the presence of reactions terms, as studied, for instance, in \cite{GildingKersner04,AudritoVazquez17,
Garriz20,Garriz23,DuGarrizQuiros25}. The profiles of these waves typically decay to zero at infinity, and consequently, these solutions lack a free boundary, which simplifies their analysis.

Even though the underlying techniques used in this paper would also apply to solutions around the Barenblatt solution, see for example \cite{Seis15,SeisTFE}, we will consider solutions that are close to traveling wave solutions, cf.\  \cite{John15,Kienzler16} as it turns out to be mathematically more convenient.

A traveling wave solution $\rho(t,y)$ is given by a profile that moves with at a constant speed $V\in \R$ in a certain direction $e\in \mathbb{S}^{n-1}$, that is
\begin{align}
    \rho(\tau,y) = P(y+\tau V e) .
\end{align}
Without loss of generality, we assume that the solution moves in the $n$th direction, that is $e=e_n$. In this case a traveling wave solution of \eqref{1} can be defined by the profile
\begin{align}
    P(y) \coloneqq \left(y_n\right)_{+}^{\frac{p-1}{m+p-3}} \quad \text{ with velocity }\quad V= \left(\frac{p-1}{m+p-3}\right)^{p-1},
\end{align}
where $\left(\cdot\right)_{+}$ denotes the positive part,
or equivalently
\begin{align}
    \rho_{tw}(\tau,y) = \left(y_n+\tau\left(\frac{p-1}{m+p-3}\right)^{p-1}\right)_{+}^{\frac{p-1}{m+p-3}}.
\end{align}
Note, that the profile $P$ is continuous due to the slow diffusion condition \eqref{2}.

The main result of this work is a regularity estimate for arbitrary derivatives of the \emph{pressure} $\rho^{\frac{m+p-3}{p-1}}$ of a solution $\rho$ of equation \eqref{1}. For the result to hold, we demand the initial pressure to be close (in a homogeneous Lipschitz sense) to the traveling wave pressure. As a consequence we deduce the analyticity of any level set and thus of the free boundary. 
The approach used in this paper is borrowed from \cite{Kienzler16}, where Kienzler deduced identical results for the porous medium equation. We extend the results to the larger class of doubly nonlinear diffusion equations.

\begin{theorem}\label{3}
    There exists an $\varepsilon>0$ such that for any nonnegative $\rho_0$ with $\rho_0^{\frac{m+p-3}{p-1}}\in \dot{C}^{0,1}\left(\R^n\right)$ satisfying
    \begin{align}\label{bedingunganfangsdaten}
        \|\nabla \rho_0^{\frac{m+p-3}{p-1}} - e_n\|_{L^\infty\left(\mathcal{P}(\rho_0)\right)} \leq \varepsilon,
    \end{align}
    there exists a solution $\rho$ of the doubly nonlinear diffusion equation \eqref{1} with initial datum $\rho_0$ satisfying
    \begin{align}
        \rho^{\frac{m+p-3}{p-1}}\in C^\infty\left(\mathcal{P}\left(\rho\right)
        \right) \quad \text{and} \quad \|\nabla \rho(t)^{\frac{m+p-3}{p-1}}-e_n\|_{L^\infty(\mathcal{P}(\rho(t)))}\leq c\varepsilon 
    \end{align}
    for a constant $c>0$ and additionally
    \begin{align}\label{smoothnessu}
        \sup \limits_{(\tau,y)\in \mathcal{P}(u)} \tau^{k+|\beta|}\left| \partial_\tau^k\partial_y^\beta \left(\nabla \rho^{\frac{m+p-3}{p-1}}(\tau,y) - e_n\right)\right| \leq \tilde{c} \|\nabla \rho_0^{\frac{m+p-3}{p-1}}-e_n\|_{L^\infty\left(\mathcal{P}\left(
            \rho_0\right)\right)},
    \end{align}
    for any  $k \in \N_0$, $\beta\in \N_0^n$ and a constant $\tilde{c} = \tilde{c}(k,\beta)>0$. Moreover, all level sets of $\rho$ are analytic.
    
\end{theorem}

Note that $e_n=\nabla\rho_{tw}^{\frac{m+p-3}{p-1}}$, so condition \eqref{bedingunganfangsdaten} for the initial datum $\rho_0$ ensures the closeness to the traveling wave solution.
Therefore, Theorem \ref{3} establishes the stability of the pressure of solutions to the doubly nonlinear diffusion equation \eqref{1} that are close to a flat traveling wave front.

In the special cases of the porous medium equation (see \cite{Kienzler16} and the references therein) and the $p$-Laplace equation (see \cite{DiBenedetto1993}), locally integrable solutions $\rho \in L^1_{loc}(\R^n)$ of \eqref{1} are uniquely determined by their initial datum $\rho_0\in L^1_{loc}(\R^n)$. 
To the authors' knowledge comparable uniqueness results for the general doubly nonlinear equation \eqref{1} are only known in smaller classes like $L^1(\R^n)$, see \cite{ZhaoYuan1995}, or $C^0(\R^n)\cap L^\infty(\R^n)$, see \cite{Kalashnikov1987}, which do not include the constructed solutions.

\subsection{New coordinates and perturbation equation}

One of the main challenges in the analysis of equation \eqref{1} is the moving free boundary $\partial \mathcal{P}$. With help of a von-Mises change of variables, we transfer the equation onto a fixed domain, namely the Euclidean half-space. Before doing so, let us consider the equation in traveling wave coordinates, that is we consider $v(\tau,y)$ given by $v(\tau, y) = \rho\left( \tau , y', y_n-V\tau\right)$ where $y=(y',y_n)$ with $y'\in \R^{n-1}$.
Then $v$ satisfies the equation
\begin{align}
    \partial_\tau v - \div \left(v^{m-1}\left|\nabla v \right|^{p-2}\nabla v\right) + V\partial_{y_n} v =0
\end{align}
and the traveling wave solution $\rho_{tw}$ becomes stationary, namely $v_{tw}( y) = \left(y_n\right)_{+}^{\frac{p-1}{m+p-3}}$.
Motivated by the form of the stationary solution $v_{tw}$, it seems more convenient to consider the pressure variable $g = v^{\frac{m+p-3}{p-1}}$, (still in traveling wave coordinates). After additionally rescaling the time $t\mapsto V\tau$, the pressure satisfies the equations
 \begin{align}
    \partial_t g - \frac{m+p-3}{p-1}g\div \left(\left|\nabla g \right|^{p-2}\nabla g\right) - \left|\nabla g\right|^p + \partial_{y_n} g=0,
 \end{align}
and the traveling wave solution becomes $g_{tw}( y) = \left(y_n\right)_{+}$.

For a function $g$ close to $g_{tw}$ in the homogeneous Lipschitz norm, we may assume for a given time $t_0$ and point $y_0$ that 
\begin{align}\label{bedingungimplicitefunction}
    \left|\partial_{y_n}g\left(t_0,y_0\right)\right| \neq 0.
\end{align}
By the implicit function theorem, there exists a function $\zeta = \zeta(t,x)$ such that
\begin{align}\label{graphequation}
    g\left(t, y', \zeta\left(t,x',x_n\right)
    \right) = x_n,
\end{align}
where $x'=y'$ and $x_n=g\left(t,y\right)$ or respectively $y_n = \zeta\left(t,x\right)$. For $\rho>0$ respectively $g> 0$, the new variable $x$ lies in  a fixed domain, namely the half space $ \mathbb{H}\coloneqq \left\{x = \left(x_1,\dots,x_n\right) \in \R^n : x_n>  0\right\}$. Note that the stationary solution $g_{tw} = \left(y_n\right)_+$ remains unchanged under this change of variables, i.e., it transforms into $\zeta_{tw}\left(x\right)=x_n$.
With help of the chain rule, we deduce from \eqref{graphequation} that
\begin{align}
    \partial_{y_n}g = \frac{1}{\partial_{x_n}\zeta} \quad \text{ and }\quad \partial_tg = -\frac{\partial_t\zeta}{\partial_{x_n}\zeta}
\end{align}
and similarly
\begin{align}
    \partial_{y_i}g = -\frac{\partial_{x_i}\zeta}{\partial_{x_n}\zeta} \quad \text{ for } i \in \left\{1,\dots,n-1\right\}.
\end{align}
Hence, we obtain for an arbitrary function $f(x)$ that
\begin{align}
    \partial_{y_i}f(x) = \partial_{x_i}f - \frac{\partial_{x_i}\zeta}{\partial_{x_n}\zeta}\partial_{x_n}f \quad \text{ and } \quad \partial_{y_n}f(x) = \frac{1}{\partial_{x_n}\zeta}\partial_{x_n}f.
\end{align}
After a further rescaling in time $t\mapsto \left(p+m-3\right)t$, the new variable $\zeta(t,x)$ satisfies
\begin{align}\mel\label{transformedpressureequation}
    \partial_t \zeta - \frac{1}{p-1}x_n \left(\termzeta\right)^{\frac{p-2}{2}} \left( \Delta'\zeta-2\frac{\nabla'\zeta\cdot \nabla'\partial_n\zeta}{\partial_n\zeta} + \termzeta \partial_n^2\zeta\right)\\&
    -\frac{1}{p-1} x_n \nabla'\zeta\cdot \nabla'\left(\left(\termzeta\right)^{\frac{p-2}{2}}\right) \\&
    +\frac{1}{p-1}x_n\termzeta \partial_n\left(\left(
        \termzeta
    \right)^{\frac{p-2}{2}}\right)\\&
    +\frac{1}{m+p-3}\left(\termzeta\right)^{\frac{p}{2}}\partial_n\zeta -\frac{1}{m+p-3} = 0 \quad \text{ in } \left[0,\infty\right) \times {\mathbb{H}}.
\end{align}
Here, we have used the notation $\grad' = (\partial_1,\dots, \partial_{n-1})^T$ for the tangential gradient and $\Delta' = \partial_1^2+\dots \partial_{n-1}^2$ for the tangential Laplacian. 

We will derive our main result Theorem \ref{3} for the pressure variable from the corresponding result for implicit variable $\zeta$, which reads as follows.
\begin{theorem}\label{Theoremv}
    There exists an $\varepsilon>0$ such that for any initial datum $
    \zeta_0 \in \dot{C}^{0,1}\left(\overline{\mathbb{H}}\right)$ satisfying
    \begin{align}
        \norm{\nabla \left(\zeta_0-\zeta_{tw} \right)}_{L^\infty\left(\mathbb{H}\right)}\leq \varepsilon
    \end{align}
    there exists a smooth solution $\zeta\in C^\infty\left((0,T)\times \overline{\mathbb{H}}\right)$ to the transformed pressure equation \eqref{transformedpressureequation} with initial datum $\zeta_0$ satisfying
    \begin{align}\label{smoothnessv}
        \sup \limits_{(t,x)\in (0,T)\times H} t^{k+|\beta|}\left| \partial_t^k\partial_x^\beta \nabla \left(\zeta(t,x) - \zeta_{tw}(x) \right)\right| \leq \tilde{c} \norm{\nabla \left(\zeta_0-\zeta_{tw}\right)}_{L^\infty(\mathbb{H})}
    \end{align}
    for any  $k \in \N_0$, $\beta\in \N_0^n$ and a constant $\tilde{c} = \tilde{c}(k,\beta)>0$.

    Moreover, the solution $\zeta$ is analytic in temporal and tangential directions. That is, there exist constants $\Lambda>0$ and $\hat{c}>0$ such that
    \begin{align}\MoveEqLeft[13]\label{analyticityv}
        \sup \limits_{(t,x)\in (0,T)\times H} t^{k+|\beta'|}\left| \partial_t^k\partial_{x'}^{\beta'} \nabla \left(\zeta(t,x) - \zeta_{tw}(x) \right) \right|\\&
         \leq \hat{c} \Lambda^{-k-|\beta'|}k!\beta'!\norm{\nabla \left(\zeta_0-\zeta_{tw}\right)}_{L^\infty(\mathbb{H})}\\&
    \end{align}
    for any $k \in \N_0$ and $\beta'\in \N_0^{n-1}$.
\end{theorem}
Similar to the previous theorem, it holds that $\nabla \zeta_{tw} = e_n$.
As a first consequence, this result justifies the use of the implicit function theorem to solve the equation \eqref{graphequation} for the implicit variable $y_n = \zeta$. Indeed, we obtain for every $(t,x)\in(0,T)\times \overline{\mathbb{H}}$ that $\left|\partial_n\zeta(t,x)-1\right|\leq \tilde{c}\varepsilon$ and thus, using $\partial_{y_n}g = (\partial_{x_n}\zeta)^{-1}$, it holds that $(1+\tilde{c}\varepsilon^{-1})\leq \partial_ng(t,y)\leq (1-\tilde{c}\varepsilon)^{-1}$. This ensures condition \eqref{bedingungimplicitefunction}.

Moreover, Theorem \ref{Theoremv} shows, that the change of variables $x\mapsto y$ is almost an isometry, what makes it a global reparametrization. For fixed $t$ it holds that
\begin{align}
    \left|(y-\hat{y})-(x-\hat{x})\right|& = \left|\left(\zeta(t,x)-\zeta(t,\hat{x})\right) - (x-\hat{x})\right| \\
    &\leq \sup \left|\nabla \zeta - e_n \right|\left| x-\hat{x} \right|\leq \tilde{c}\varepsilon\left|x-\hat{x}\right|,
\end{align}
since $\zeta_{tw}(x) = x_n$. This immediately yields
\begin{align}
    \left(1-\tilde{c}\varepsilon\right)|x-\hat{x}|\leq |y-\hat{y}| \leq \left(1+\tilde{c}\varepsilon\right)|x-\hat{x}|.
\end{align}
Hence, the change of variables $\overline{\mathbb{H}}\ni x \mapsto y\in {\mathcal{P}}\left(g\right)$ is an injective quasi-isometry and a solution $\zeta$ of \eqref{transformedpressureequation} provides a \emph{nonnegative}  solution $\rho$ of the doubly nonlinear diffusion equation \eqref{1} by reversing the performed transformation. Furthermore, the estimates $\eqref{smoothnessv}$ on the derivatives of $\zeta$ carry over to the corresponding estimates \eqref{smoothnessu} for the pressure $\rho^{\frac{m+p-3}{p-1}}$. 

Note, that the level set $\left\{(t,y) : g(t,y) = \lambda\right\}$ of $g$ at height $\lambda$ is given by
\begin{align}
    \left\{\left(t,x',y_n\right) : y_n = 
    \zeta\left(t,x',\lambda\right)\right\}.
\end{align}
Therefore, the analyticity of $\zeta$ in time and the tangential directions, cf.\ \eqref{analyticityv}, implies the analyticity of all level sets of $g$ respectively $\rho^{\frac{m+p-3}{p-1}}$.
Thus, our main result Theorem \ref{3} is a direct consequence of Theorem \ref{Theoremv}.

Let us now consider a small perturbation $w$ around the stationary solution $\zeta_{tw}=x_n$, that is assuming $\zeta=x_n+w$. After a rescaling of  the first $n-1$ coordinates, namely $x\mapsto z = \left(z',z_n\right)=\left((p-1)^{1/2}x',x_n\right)$, the evolution of the perturbation $w(t,z)$ satisfies
\begin{align}\mel\label{perturbationequation}
    \partial_tw +\L_{\sigma}w = \frac{1}{p-1}\left\{\left(\rest{p-2}-1\right)z_n\Delta'w + \left(\rest{p}-1\right) z_n\partial_n^2w
    \right.\\&\left.
    +\frac{p-1}{m+p-3}\left(1-\rest{p}\right)\partial_n w + \frac{p-1}{m+p-3}\left(1-\rest{p}-p\partial_nw\right)\right.\\&\left.
    -2z_n \rest{p-2}\frac{\nabla'w\cdot \nabla'\partial_nw}{1+\partial_nw} + z_n  \nabla'\rest{p-2}\cdot \nabla'w\right.\\&\left.
    -z_n \rest{2}\left(\partial_n \rest{p-2} + (p-2)\partial_n^2w\right)\right\} \eqqcolon \mathcal{N}[w] \quad \text{ in } \left[0,\infty\right) \times \overline{\mathbb{H}},
\end{align} 
where the linear operator $\L_\sigma$ is given by 
\begin{align}\L_\sigma w \coloneqq -z_n^{-\sigma}\nabla \cdot \left(z_n^{\sigma+1}\nabla w\right) \quad \text{with} \quad \sigma = \frac{2-m}{m+p-3},
\end{align}    
and we used the short notation
\begin{align}
    \rest{q}\coloneqq \left(\termw\right)^{q/2}.
\end{align}    
Note that $\sigma >-1$, provided that $p>1$ and $m+p>3$, that is in the slow diffusion regime \eqref{2}.

We investigate this perturbation equation because a sufficiently regular solution $w$ of \eqref{perturbationequation} yields to an accordingly regular solution $v = x_n+w$ of the transformed pressure equation \eqref{transformedpressureequation}. Thus, Theorem \ref{Theoremv} is a direct consequence of the corresponding result for the perturbation variable, see Theorem \ref{maintheorem} below.

The linear operator $\L_\sigma$ and the associated degenerate parabolic linear equation
\begin{align}\label{linearproblem}
    \partial_t w + \L_\sigma w=f \quad \text{ on } [0,\infty)\times {\mathbb{H}}
\end{align}
are well understood and already appeared in a similar treatment of the porous medium equation or the thin film equation, see \cite{John15,Kienzler16,Seis15,SeisTFE,ChoiMcCannSeis23}.

In order to formulate our main result for the perturbation variable $w$ we have to introduce some notation. 
For its proof, we will make use of a fixed point argument. An idea based Angenent's trick \cite{angenent_1990} and an extension by Koch and Lamm \cite{koch2012} enables us to derive the analyticity statements. To perform the fixed point argument, we have to work with cleverly chosen Banach spaces. In \cite{Kienzler16}, Kienzler introduced suitable Banach spaces $X$ and $Y$ (that we will define later), where a maximal regularity estimate of the form
\begin{align}
    \norm{w}_{\dot{C}^{0,1}} + \norm{w}_X \lesssim \norm{f}_Y + \norm{w_0}_{\dot{C}^{0,1}}
\end{align} for solutions of the linear equation \eqref{linearproblem} with initial datum $w_0$ holds true. The analysis in \cite{Kienzler16} is inspired by Koch's and Lamm's approach to treat semilinear parabolic equations \cite{koch2012}. Similar to Koch's and Tataru's idea to tackle the Navier-Stokes equations \cite{KochTataru2001} using a Carleson measure formulation of the BMO norm, Koch and Lamm used Gaussian decay estimates and Calderon--Zygmund theory to bound solutions in specific Carleson measures.
In contrast to the equations treated in \cite{KochLamm15}, \eqref{linearproblem} fails to be strictly parabolic at the boundary $\left\{z_n=0\right\}$. We overcome this problem by replacing the flat Euclidean metric on the half space by a Carnot--Carath\'eodory distance $d$ arising from the linear operator $\L_\sigma$. This distance has been proven useful in context of porous medium equation, see \cite{DaskalopoulosHamilton1998,KochHabilitation} for early contributions. It is equivalent to the distance
\begin{align}
\tilde{d}(z,z')\coloneqq\frac{|z-z'|}{\sqrt{z_n}+\sqrt{z'_n}+\sqrt{|z-z'|}}.
\end{align}
Reflecting the degeneracy of the equation, it is singular towards the boundary $\left\{z_n=0\right\}$. For further details we refer to \cite{Kienzler16} or \cite{Seis15}.

Incorporating these considerations, Kienzler introduced the (semi-)norms
\begin{align}
\|f\|_{Y(q)} = \sum\limits_{|\beta| \in \left\{0,1\right\}}\sup \limits_{\substack{0<r^2<T\\\hat{z}\in \mathbb{H}}} r^2\left(\frac{1}{r\left(r+\sqrt{\hat{z}_n}\right)}\right)^{1-|\beta|} \left|Q_r(\hat{z})\right|^{-\frac{1}{q}}\|\partial_z^{\beta}f\|_{L^q\left(Q_r(\hat{z})\right)} 
\end{align} 
and
\begin{align}\mel
    \|w\|_{X(q)} =\|\nabla w\|_{L^\infty\left((0,T)\times H\right)}+\|\sqrt{t}\sqrt{z_n}\nabla^2 w\|_{L^\infty\left((0,T)\times H\right)} \\&
    + \sum \limits_{\left(\ell,k,|\beta|\right) \in E} \sup \limits_{\substack{0<r^2<T\\ \hat{z}\in \mathbb{H}}} r^2\left(\frac{r}{r+\sqrt{\hat{z}_n}}\right)^{2\ell-|\beta|}\left|Q_r(\hat{z})\right|^{-\frac{1}{q}} \|z_n^{\ell}\partial_t^k\partial_z^\beta \nabla w\|_{L^q\left(Q_r(\hat{z})\right)},
\end{align}
where $E=\left\{(0,0,1),(0,1,0),(1,0,2)\right\}$ and $Q_r(\hat{z}) \coloneqq \left(\frac12 r^2,r^2\right)\times B^d_r(\hat{z})$. Note  that the balls $B_r^d(\hat{z})$ are defined with respect to the Carnot--Carath\'eodory distance $d$. The corresponding Banach spaces $X(q)$ and $Y(q)$ are defined in the canonical way. 

Let us now precisely phrase the result for the perturbation variable $w$ that parallels the results found for the porous medium equation in \cite{Kienzler16,Seis15} or the thin film equation in \cite{John15,SeisTFE}.
\begin{theorem}\label{maintheorem}
    Fix $q$ with $\max\left\{2(n+1),(1+\sigma)^{-1}\right\}<q<\infty$. Then, there exist $\varepsilon_0>0$ and $\tilde{C}>0$ such that for any $\varepsilon \leq \varepsilon_0$ and any initial datum $
    w_0 \in \dot{C}^{0,1}\left(\overline{\mathbb{H}}\right)$ satisfying
    \begin{align}
        \|\nabla w_0\|_{L^\infty\left(\mathbb{H}\right)}\leq \varepsilon
    \end{align}
    there exists a smooth unique solution $w$ to the perturbation equation \eqref{perturbationequation} in the class $\norm{w}_{X(q)}\leq \tilde{C}\varepsilon$. This solution satisfies
    \begin{align} \label{6}
        \sup \limits_{(t,z)\in (0,T)\times \mathbb{H}} t^{k+|\beta|}\left| \partial_t^k\partial_z^\beta \nabla w(t,z) \right| \leq \tilde{c} \norm{\nabla w_0}_{L^\infty\left(\mathbb{H}\right)},
    \end{align}
    for any  $k \in \N_0$, $\beta\in \N_0^n$ and a constant $\tilde{c} = \tilde{c}(k,\beta)>0$.
    
    Moreover, the solution depends analytically on its initial datum $w_0$ and is analytic in temporal and tangential directions. That is, there exist constants $\Lambda>0$ and $\hat{c} >0$ such that
    \begin{align}\mel\label{t3}
        \sup \limits_{(t,z)\in (0,T)\times \mathbb{H}} t^{k+|\beta'|}\left| \partial_t^k\partial_{z'}^{\beta'} \nabla w(t,z) \right| \leq \hat{c} \Lambda^{-k-|\beta'|}k!\beta'!\norm{\nabla w_0}_{L^\infty(\mathbb{H})}\\&
    \end{align}
    for any $k \in \N_0$ and $\beta'\in \N_0^{n-1}$.
\end{theorem}

Note that this Theorem \ref{maintheorem} coincides with Theorem 4.1 in \cite{Kienzler16}, which covers the same result for the special case of the porous medium equation.
The solution $w$ found in the above theorem directly generates a solution $\zeta = x_n + w$ of the transformed pressure equation \eqref{transformedpressureequation} verifying Theorem \ref{Theoremv}. 
Therefore, the proof of Theorem \ref{maintheorem}, see Chapter \ref{proofs}, will close the argument.
 
\section{Derivation of Theorem \ref{maintheorem}}\label{proofs}

To clarify the approach that will be used to derive Theorem \ref{maintheorem}, we consider the situation in a slightly more abstract way.
We are concerned with an initial value problem of the form
\begin{align}\label{nichtlinearesproblem}
    \begin{cases}
        \partial_t w + \L w &= \mathcal{N}[w] \quad \text{ in } (0,T)\times \mathbb{H}\\
        w(0,\cdot) &= g \qquad\quad \text{ in }\mathbb{H},
    \end{cases}
\end{align}
where   $\L$ is a linear spatial differential operator and the possibly nonlinear term $\mathcal{N}[w]$ depends analytically on $w$ near the trivial solution $w\equiv 0$.

We collect some assumptions that will be sufficient to deduce suitable results concerning existence, uniqueness and regularity of solutions of \eqref{nichtlinearesproblem} comparable to Theorem \ref{metatheorem}.
For this, let $(X_0,\|\cdot\|_{X_0})$, $(X_1,\|\cdot\|_{X_1})$ and $(Y,\|\cdot\|_Y)$ be three function spaces on $\mathbb{H}$ such that $X\coloneqq X_0\cap X_1$ (equipped with the norm $\|\cdot\|_X \coloneqq\|\cdot\|_{X_0}+\|\cdot\|_{X_1}$) and $Y$ are Banach spaces. The needed assumptions are the following.
\begin{enumerate}[label=(A\arabic*)]
    \item\label{assumption1} For any initial datum $g\in X_0$ and $f\in Y$, the linear initial value problem 
    \begin{align}\label{4}
        \begin{cases}
            \partial_t w + \L w &= f\\
            w(0,\cdot) &= g 
        \end{cases}
    \end{align}
    has a unique solution $w$ that satisfies
    \begin{align}
        \|w\|_X\lesssim \|f\|_Y+\|g\|_{X_0}.
    \end{align}
    \item \label{assumption2}The nonlinearity $\mathcal{N}$ is   near $w=0$ analytic from $X$ to $Y$. Moreover, for $w_1$, $w_2$ satisfying $\|w_1\|_X\leq R$ and $\|w_2\|_X\leq R$, with $0<R\ll1$ it holds that
    \begin{align}
        \|\mathcal{N}[w]\|_Y&\lesssim \|w\|_X^\alpha \quad \text{ for some }\alpha >1 \\\
        \text{ and } \quad \|\mathcal{N}[w_1]-\mathcal{N}[w_2]\|_Y&\lesssim R \|w_1-w_2\|_X.
    \end{align}
    \item \label{assumption3}The linear operator $\L$ satisfies $\|\L w\|_Y\lesssim  \|w\|_X$.
    \item \label{assumption4}The linear operator $\L$ is invariant under translation in tangential directions  and $\|  \nabla' w\|_Y\lesssim \|w\|_X$.
\end{enumerate}
The upcoming theorem presents which implications the assumptions \ref{assumption1} to \ref{assumption4} permit for the nonlinear problem \eqref{nichtlinearesproblem}.
\begin{theorem}\label{metatheorem}
    Provided that the two assumptions \ref{assumption1} and \ref{assumption2} hold true, there exists an $\varepsilon_0>0$ and $\tilde{C}>0$ such that for any $\varepsilon \leq \varepsilon_0$ and any initial datum $g$ satisfying
\begin{align}
    \|g\|_{X_0} \leq \varepsilon,
\end{align}
there exists a unique solution $w$ to the (nonlinear) initial value problem \eqref{nichtlinearesproblem} in the class $\|w\|_X\leq \tilde{C}\varepsilon$. Moreover, the solution depends analytically on its initial datum $g$.
If additionally assumption \ref{assumption3} is fulfilled, the solution $w$  is also analytic in time,  satisfying 
\begin{align}\label{t1}
     \|t^k \partial_t^k w\|_X\leq c_1 \Lambda_1^{-k}k! \|g\|_{X_0}
\end{align}
for any $k\in \N_0$ and constants $c_1>0$ and $\Lambda_1>0$ independent of $k$.

Similarly, assumption \ref{assumption4} yields analyticity in the translation-invariant directions, that is, 
\begin{align}\label{t2}
     \|t^{k+|\beta'|} \partial_{z'}^{\beta'} w\|_X\leq c_2 \Lambda_2^{-|\beta'|}\beta'! \|g\|_{X_0}
\end{align}
for any multiindex $\beta'\in \N_0^{n-1}$ and constants $c_2$ and $\Lambda_2>0$ independent of $k$.
\end{theorem}

The proof of this theorem relies on a combination of the analytic implicit function theorem and Angenent's trick, and has successfully used in various works, for instance, \cite{John15,Kienzler16,koch2012,Seis15,SeisTFE,SeisWinkler22}. Here, we provide a more abstract version.

\begin{proof}
\emph{Existence and uniqueness.}     For any $g\in X_0$ and $h\in X$, we denote the solution $ w$ to the linear problem \eqref{4}, whose existence is guaranteed by \ref{assumption1}, by $ w = I(g,h)$. It satisfies the a priori estimate
\[
\|w\|_X \lesssim \|\mathcal{N}(h)\|_Y +\|g\|_{X_0}.
\]
By \ref{assumption2}, if $\|h\|_X\le R$ for some $R\ll1$, the latter turns into 
\[
\|w\|_X\lesssim \|h\|_X^{\alpha} +\|g\|_{X_0}.
\]
Therefore, if $R\ll1$ is given and $\eps$ is sufficiently small, we deduce that $\|w\|_{X}\le R$. This shows for and $g\in B_{\eps}^{X_0}(0)$, the solution map $I(g,\cdot)$ is a continuous  mapping from $ \overline{B_R^X(0)}$ to $ \overline{B_R^X(0)}$.

The solution map is also contracting. To see this, we consider two solutions $w$ and $\tilde w$ associated to the linear problem with nonlinearities generated by $h$ and $\tilde h$ in $\overline{B_R^X(0)}$, that is, $w=I(g,h)$ and $\tilde w =I(g,\tilde h)$. Thanks to the a priori estimate from \ref{assumption1} and the Lipschitz property in \ref{assumption2}, we have that
\[
\|I(g,w) - I(g,\tilde w)\| = \|w-\tilde w\|_{X} \lesssim \|\mathcal{N}(h) - \mathcal{N}(\tilde h)\|_{Y} \lesssim R\|h-\tilde h\|_{X}.
\]
Thus, the solution map is a contraction if $R$ is sufficiently small. Fixing such an $R$, the contraction theorem thus yields the existence of a unique fixed point $w = I(g,w)$ in $\overline{B_R^X(0)}$, solving thus the non-linear problem \eqref{nichtlinearesproblem}.

\medskip

\noindent

\emph{Analytic dependece on initial data.} We will establish the analytic dependence on the initial data by an application of the analytic implicit function theorem, see, e.g.~Theorem 15.3 in \cite{Deimling85}. We start by noticing because the nonlinearity $\mathcal{N}(w)$ is an analytic function by the virtue of \ref{assumption2},   the contraction map $I$ is analytic on $B_{\eps}^{X_0} (0) \times B_R^X(0) $. We consider $J(g,w) = w - I(g,w)$, which is then also analytic. Moreover, because $J(0,0)=0$ and $D_wJ(0,0)=\mathrm{id_X}$, the implicit function theorem yields the existence of radii $\tilde R\le R$ and $\tilde \eps\le \eps$ and of an analytic map $A: B_{\tilde \eps}^{X_0}(0)\to B_{\tilde R}^X(0)$ such that $J(g,w)=0$ precisely if $A(g)=w$. By the uniqueness of the fixed point solution, $w=I(g,w)$, and the definition of $J$, the solution to the nonlinear problem \ref{nichtlinearesproblem}, considered as a function from $B_{\tilde \eps}^{X_0}(0)$ to $B_{\tilde R}^X(0)$, depends thus analytically on the initial datum $g$.

\medskip

\noindent

\emph{Analytic dependence on time and tangential variables.} Analyticity in time and tangential variables is obtained by what is commonly known as Angenent's trick \cite{angenent_1990}. We follows closely a version by Koch and Lamm \cite{koch2012}. For $\lambda\in \R$ and $y' \in \R^{n-1}$ given, we define 
\[
w_{\lambda,y'} = w\circ \chi_{\lambda,y'},\quad \chi_{\lambda,y'}(t,z) = (\lambda t,z'+ty',z_n).
\]
Then $w_{\lambda,y'}$ solves the nonlinear equation
\begin{equation}
\label{5}
\partial_t w_{\lambda,y'} +\L w_{\lambda,y'} = \mathcal{N}_{\lambda,y'}(w_{\lambda,y'}),
\end{equation}
where the new inhomogeneity is now
\[
\mathcal{N}_{\lambda,y'}(w_{\lambda,y'}) = \lambda\mathcal{N}(w_{\lambda,y'}) +(1-\lambda)\L w_{\lambda,y'} +y'\cdot \grad'w_{\lambda,y'}.
\]
We notice that $\mathcal{N}_{1,0} =\mathcal{N}$. 

Analogously to the above, we write $I_{\lambda,y'}(g,h)$ for the solution to the linear solution of  \eqref{4} with initial datum $g$ and inhomogeneity $\mathcal{N}(h)$. Notice that for the existence, we exploit \ref{assumption3} and \ref{assumption4}. We furthermore set $J_{\lambda,y'}(g,h) = h-I_{\lambda,y'}(g,h)$. It is then readily checked that $J_{1,0}(0,0)=0$ and $D_h J_{1,0}(0,0) = \mathrm{id}_X$. Therefore, thanks to the analytic implicit function theorem, there exist small constants $\delta>0$, $\hat R\le R$ and $\hat \eps\le \eps$, and an analytic function $A_{\lambda,y'}(g) = A(\lambda,y',g)$ from $B_{\delta}^{\R}(1)\times B_{\delta}^{\R^{n-1}}(0)\times B_{\hat \eps}^{X_0}(0)$ to $B_{\hat R}^X(0)$, such that $J_{\lambda,y'}(g,A_{\lambda,y'}(g))=0$. In particular, because $w=A(g)$ was the unique solution to the original problem \eqref{nichtlinearesproblem}, it holds that $A_{\lambda,y'}(g) = A(g)\circ\chi_{\lambda,y'} = w_{\lambda,y'}$. It follows that the solution $w_{\lambda,y'}$ to the transformed nonlinear problem \eqref{5} depends analytically near $(1,0)\in\R\times \R^{n-1}$ on the variables $(\lambda,y')$. In particular, there exists a constant $\Lambda$ such that
\[
\|\partial_{\lambda}^k\partial_{y'}^{\beta'} w_{\lambda,y'}|_{(\lambda,y') = (1,0)} \|_X \le \Lambda^{-k-|\beta'|} k!\beta'! \|g\|_{X_0},
\]
for any $k\in \N_0$ and $\beta'\in\N_0^{n-1}$. However, via the defintion of $w_{\lambda,y'}$, the latter turns into
\[
 \|t^{k+|\beta'|} \partial_t^k\partial_{z'}^{\beta'} w(t,z)\|_{X}\le \Lambda^{-k-|\beta'|}k!\beta'! \|g\|_{X_0},
\]
which gives the desired estimate with for the time and tangential derivatives. 
\end{proof}

Note that the estimates \eqref{t1} and \eqref{t2}   immediately provide the according estimate \eqref{t3} in Theorem \ref{maintheorem} if we choose the right spaces $X_0 = \dot{C}^{0,1}$, $X=X(q)$ and $Y=Y(q)$ introduced in the previous section.

With Theorem \ref{metatheorem} at hand, checking the assumptions \ref{assumption1} to \ref{assumption4} for the perturbation equation \eqref{perturbationequation} yields all statements from Theorem \ref{maintheorem} exept for  the smoothness of the solution in the transversal direction. We adress this later.  

Assumption \ref{assumption1} is covered by the following well-posedness result and maximal regularity estimate for the linear operator $\L_\sigma$ proved by Kienzler while  investigating the porous medium equation. (Kienzler proved the existence and uniqueness of suitably defined weak solutions of \eqref{linearproblem}. Since the exact definition of weak solutions is not important in our situation, we refer to \cite{Kienzler16} for further details.)

\begin{propositionA}[\cite{Kienzler16}]\label{TheoremLinearesProblem}
    Let $\max\left\{2(n+1),(1+\sigma)^{-1}\right\}<q<\infty$ and $f\in Y(q)$. Then there exists a unique weak solution $w$ of the linear perturbation equation \eqref{linearproblem} with initial datum $w_0\in \dot{C}^{0,1}\left(\overline{\mathbb{H}}\right)$ and inhomogeneity $f$. Moreover, the solution $w$ satisfies
    \begin{align}
        \norm{w}_{X(q)} \leq c \left(\norm{\nabla w_0}_{L^\infty(\mathbb{H})} + \norm{f}_{Y(q)}\right)
    \end{align}
    for a constant $c=c(\sigma,q)>0$.
\end{propositionA}

In contrast to the linear operator $\L_\sigma$, the nonlinearity $\mathcal{N}[w]$ in the perturbation equation \eqref{perturbationequation} appears to be far more complicated  compared to the porous medium equation case considered in \cite{Kienzler16}. However, a closer examination of the nonlinearity reveals enough (essentially quadratic) structure to establish assumption \ref{assumption2} similar to \cite[Lemma 2.3]{Kienzler16}.

\begin{lemma}\label{LemmaNichtlinearitaet}
    For $1\leq q \leq \infty$ and $0<R<1/2$ the mapping
    \begin{align}
        B_R^X(0) \ni w\mapsto \mathcal{N}[w]\in  Y(q)
    \end{align}
    is analytic and there exists a constant $c=c(n,q)>0$ such that
    \begin{align}
        \|\mathcal{N}[u]\|_{Y(q)} \leq c \norm{w}^2_{X(q)} \quad \text{ for all } w \in B_R^X  (0),
    \end{align}
    and
    \begin{align}
        \|\mathcal{N}[w_1]-\mathcal{N}[w_2]\|_{Y(q)} \leq c R \norm{w_1-w_2}_{X(q)} \quad \text{ for all } w_1,w_2\in B_R^X(0).
    \end{align}
\end{lemma}
\begin{proof}
    We start by shortly addressing the analyticity of the nonlinearity $\mathcal{N}$. Since $\mathcal{N}$ is a quotient of algebraic functions, it is analytic away from its poles, which is located at $1+\partial_n w=0$.
    The proof of the stated estimates proceeds in the same way as the one of the corresponding lemma for the porous medium equation, see \cite[Lemma 2.3]{Kienzler16}. We limit ourselves to show that the nonlinear terms in \eqref{perturbationequation} satisfy the following bounds, from where the proof can be easily reproduced. That is, for $\norm{\nabla w}_{L^\infty}\ll 1$ it holds that
    \begin{align}\label{bedingungN}
        \left|\mathcal{N}[w]\right|\lesssim \left|\nabla w\right|^2 + z_n \left|\nabla w\right|\left|\nabla^2w\right|
    \end{align}
    and 
    \begin{align}\label{bedingungNablaN}
        \left| \nabla \mathcal{N}[w]\right| \lesssim \left|\nabla w \right| \left| \nabla^2w\right| + z_n \left|\nabla w\right|\left|\nabla^3w\right| +z_n \left|\nabla^2w\right|^2,
    \end{align}
where  $\grad^2w$ and $\grad^3 w$ are the matrices of all second and third order derivatives, respectively.    All occurring terms except the last term on the right-hand side of \eqref{bedingungNablaN} can be easily controlled in the $X(q)$-norm.
    The remaining term can be interpolated via
    \begin{align}
        \|z_n |\nabla^2w|^2\|_{L^p(Q_r(\hat{z}))} \lesssim \|\nabla w\|_{L^\infty}\|z_n\nabla^3w\|_{L^p(Q_r(\hat{z}))},
    \end{align}
    with help of a weighted Gagliardo-Nirenberg interpolation inequality, see \cite{Kienzler13} for further details.
    To derive the bounds \eqref{bedingungN} and \eqref{bedingungNablaN}, we note that the following identities hold true:
        \begin{align}
             \rest{q} &= \frac{1}{\left(1+\partial_nw\right)^q} + \mathcal{O}\left(\left|\nabla'w\right|^2\right)
            =1 - q\partial_nw + \mathcal{O}\left(\left|\nabla w\right|^2\right) ,\label{EstimateRest}
            \\
            \partial_j\rest{q} &= -q\partial_j\partial_n w + \partial_j\partial_nw\mathcal{O}\left(\left|\nabla w\right|\right) + \partial_j\nabla'w\mathcal{O}\left(\left|\nabla w \right|\right) ,\label{EstimateErsteAbleitungRest}
        \end{align}
for any $ j=1,\dots,n$.    Using the above considerations, it immediately follows that the linear terms cancel out suitably and thus, the nonlinearity $\mathcal{N}[w]$ fulfills condition \eqref{bedingungN}.

    Next, we check the structure of the gradient of the nonlinearity. A slightly tedious but straightforward computation yields
    \begin{align}\MoveEqLeft\label{EstimateZweiteAbleitungRest}
        \partial_i\partial_j \rest{q} = \left(1+\mathcal{O}\left(\left|\nabla w\right|\right)
         \right) \\&
         \left( -q \partial_i\partial_j\partial_nw + q(q+1) \partial_i\partial_nw\partial_j\partial_nw + q \partial_i\nabla'w \partial_j\nabla'w + q \nabla'w\partial_i\partial_j\nabla'w\right) \\&+ \mathcal{O}\left(\left|\nabla w\right|\right) \left(\partial_j\nabla
        'w\partial_i\partial_nw + \partial_i\nabla'w\partial_j\partial_nw + \partial_i\nabla'w\partial_j\nabla'w\right).
    \end{align}
    With help of \eqref{EstimateRest}, \eqref{EstimateErsteAbleitungRest} and \eqref{EstimateZweiteAbleitungRest} we eventually deduce that the $\nabla \mathcal{N}[w]$ satisfies the desired condition \eqref{bedingungNablaN}.
    For the remainder of the proof we refer to \cite{Kienzler16}.
\end{proof}

The assumptions \ref{assumption3} and \ref{assumption4}  can be checked readily. Therefore, it only remains to prove smoothness in the transversal direction.

\begin{theorem}
Supose that $w$ is the solution to the perturbation equation \eqref{perturbationequation} that is analytic in temporal and tangential variables and that satisfies the decay estimates \eqref{t1} and \eqref{t2}. Then the solution is smooth in transversal direction and \eqref{6} holds. 
\end{theorem}

\begin{proof}Bounds on transversal derivatives can be obtained by using that the linear operator satisfies the commutation rule
\[
\partial_n \L_{\sigma} w = \L_{\sigma+1}\partial_n w - \Delta'w,
\]
and thus, it holds
\[
\partial_t (t\partial_n w) + \L_{\sigma+1} (t\partial_n w)  = t\partial_n \mathcal{N}(w) - t \Delta' w + \partial_n w.
\]
The nonlinearity can be estimated in a similar way as in Lemma  \ref{LemmaNichtlinearitaet}. Here, we will only establish the uniform control of the second order derivatives, which can be obtained without performing this tedious approach. Higher order transversal derivatives can be controlled   via a suitable iteration.

We use the following  Morrey-type inequality to carry the bounds in \eqref{t1} over to those in \eqref{6}:
For any $q>n$, it holds that
\begin{align*}
\|v\|_{L^{\infty}(Q_r^d(\hat z))} & \lesssim (r+\sqrt{\hat z_n})^{-2} \|z_n v\|_{L^q(Q_r^d(\hat z))} + r(r+\sqrt{\hat z_n})^{-1} \|z_n \grad v\|_{L^q(Q_r^d(\hat z))} \\
&\quad + r^2 (r+\sqrt{\hat z_n})^{-2} \|z_n \partial_t v\|_{L^q(Q_r^d(\hat z))}.
\end{align*}
This inequality can be derived similarly to the Euclidean case. For $v=\partial_n^2 w$, it implies 
\begin{align*}
r^2 \|\partial_n^2 w\|_{L^{\infty}(Q_r^d(\hat z))}& \lesssim r^2 (r+\sqrt{\hat z_n})^{-2} \|z_n \partial_n^2 w\|_{L^q(Q_r^d(\hat z))}\\
&\quad  + r^3(r+\sqrt{\hat z_n})^{-1} \|z_n \grad \partial_n^2 w\|_{L^q(Q_r^d(\hat z))} \\
&\quad + r^4 (r+\sqrt{\hat z_n})^{-2} \|z_n \partial_t \partial_n^2 w\|_{L^q(Q_r^d(\hat z))}.
\end{align*}
In view of the definition of the $X(q)$ norms and because $z_n \lesssim (r+\sqrt{\hat z_n})^2$ for any $z \in B_r^d(\hat z)$, this bound implies
\[
r^2 \|\partial_n^2 w\|_{L^{\infty}(Q_r^d(\hat z))} \lesssim \frac{r}{r+\sqrt{\hat z_n}} \left(\|w\|_{X(q)} + \|t\partial_t w\|_{X(q)}\right).
\]
For $t\in (r^2/2,r^2)$, by using the decay estimates in \eqref{t1}, we deduce that 
\[
\|t\partial_n^2 w\|_{L^{\infty}(Q_r^d(\hat z))} \lesssim \|g\|_{X_0}.
\]
Maximizing in $r$ and $\hat z$ gives the desired statement.
\end{proof}

\section*{Acknowledgment}
This work is funded by the Deutsche Forschungsgemeinschaft (DFG, German Research Foundation) under Germany's Excellence Strategy EXC 2044 --390685587, Mathematics M\"unster: Dynamics--Geometry--Structure.

\section*{Declaration of interests}
The authors do not work for, advise, own shares in, or receive funds from any organization that could benefit from this article, and have declared no affiliation other than their research organizations.

\bibliography{mybib}
\bibliographystyle{abbrv}
\end{document}